\documentclass[review]{elsarticle}
\usepackage{lineno,hyperref}
\usepackage{amsmath,amssymb}
\usepackage{amsfonts}
\usepackage{verbatim}
\usepackage{url}
\usepackage{graphicx}
\usepackage{float}
\usepackage{mathrsfs}
\usepackage{epstopdf}
\usepackage{cancel}
\usepackage{MnSymbol}
\usepackage{tikz}

\allowdisplaybreaks
\newtheorem{theorem}{Theorem}

\newtheorem{corollary}[theorem]{Corollary}
\newtheorem{example}{Example}
\newtheorem{lemma}[theorem]{Lemma}

\newproof{proof}{Proof}
\numberwithin{equation}{section}
\numberwithin{theorem}{section}

\begin{document}

\begin{frontmatter}

\title{Oriented Hypergraphic Matrix-tree Type Theorems and Bidirected Minors via Boolean Order Ideals.}

\author[]{Ellen Robinson}
\author[]{Lucas J. Rusnak\corref{mycorrespondingauthor}}
\ead{Lucas.Rusnak@txstate.edu}
\author[]{Martin Schmidt}
\author[]{Piyush Shroff}

\address{Department of Mathematics\\
Texas State University\\
San Marcos, TX 78666, USA}

\cortext[mycorrespondingauthor]{Corresponding author}

\begin{abstract}

Restrictions of incidence-preserving path maps produce an oriented hypergraphic All Minors Matrix-tree Theorems for Laplacian and adjacency matrices. The images of these maps produce a locally signed graphic, incidence generalization, of cycle covers and basic figures that correspond to incidence-$k$-forests. When restricted to bidirected graphs the natural partial ordering of maps results in disjoint signed boolean lattices whose minor calculations correspond to principal order ideals. As an application, (1) the determinant formula of a signed graphic Laplacian is reclaimed and shown to be determined by the maximal positive-circle-free elements, and (2) spanning trees are equivalent to single-element order ideals.

\end{abstract}

\begin{keyword}
Matrix-tree theorem \sep Laplacian matrix \sep signed graph \sep bidirected graph \sep oriented hypergraph
\MSC[2010] 05C50 \sep 05C65 \sep 05C22
\end{keyword}

\end{frontmatter}



\section{Introduction}

An oriented hypergraph is a signed incidence structure that first appeared in \cite{Shi1} to 
study applications to VLSI via minimization and logic synthesis, and generalize Kirchhoff's laws \cite{ShiBrz, Shi2}. 
This incidence based approach allows for Laplacian, spectral, and balanced, graph-theoretic theorems to be extended to hypergraphs using their locally signed graphic structure \cite{OHHar, Reff1, AH1}.  Moreover, the concepts of a balanced hypergraph \cite{Berge2, CO1} and a balanced $\{0,\pm1\}$-matrix \cite{DBM, BM, FHO} can simultaneously be studied through their oriented hypergraphic structure \cite{OH1}.
 
Sachs' Theorem \cite{SGBook} characterizes the coefficents of the characteristic polynomial of the adjacency matrix by generalizing the concept of a cycle cover of a graph; and has long known application in the study of molecular orbitals \cite{SachsChem}. Sachs' Theorem was recently generalized to signed graphs in \cite{Sim1}, and to oriented hypergraphs in \cite{OHSachs}. We show that the incidence generalization of cycle covers (called \emph{contributors}) obtained in \cite{OHSachs} allows for the hypergraphic generalization of the All Minors Matrix-tree Theorem of Chaiken in \cite{Seth1, Seth2}. Moreover, when restricted to bidirected graphs, the contributors align with Chaiken's $k$-forests.

Section \ref{MTT} collects the necessary oriented hypergraphic background to present the All Minors Matrix-tree theorem as a consequence of the oriented hypergraphic Sachs' Theorem in \cite{OHSachs} --- the coefficients of the characteristic polynomial are the diagonal minors, are determined by contributor sums, and a similar proof can be used to obtain any minor. 

Contributor maps are specialized to Laplacians of bidirected graphs in Section \ref{CMap}. A
natural partial ordering of contributors is introduced where each associated
equivalence classes (called \emph{activation classes}) is boolean; this is
done by introducing \emph{incidence packing} and \emph{unpacking}
operations. If an oriented hypergraph contains edges of size larger than $3$%
, then unpacking is not well-defined, and the resulting equivalence classes
need not be lattices. Activation classes are further refined via iterated
principal order ideals in order to examine minors of the Laplacian.

Section \ref{MTT2} examines the contributors in the adjacency completion of
a bidirected graph to obtain a restatement of the All Minors Matrix-tree
Theorem in terms of sub-contributors (as opposed to restricted
contributors). This implies there is a universal collection of contributors
(up to resigning) which determines the minors of all bidirected graphs that
have the same injective envelope --- see \cite{Grill1} for more on the
injective envelope. These sub-contributors determine permanents/determinants
of the minors of the original bidirected graph and are activation equivalent
to the forest-like objects in \cite{Seth1}. Additionally, the standard
determinant of the signed graphic Laplacian is presented as a sum of maximal
contributors, while the first minors of the Laplacian contain a subset of
contributors that are activation equivalent to spanning trees.

While the techniques introduced for bidirected graphs do not readily extend
to all oriented hypergraphs, they bear a remarkable similarity to Tutte's
development of transpedances in \cite{Tutte} --- indicating the possibility
of a locally signed graphic interpretation of transpedance theory. Since
contributor sets produce the finest possible set of objects signed $\{0,\pm
1\}$ whose sum produces the permanent/determinant it is natural to ask what
classes of graphs achieve specific permanent/determinant values. Additionally, recent work on vertex-weighted Laplacians \cite{ChungF1}, graph dynamics \cite{SpecDyn}, and oriented spanning trees and sandpile groups \cite{LL}, seem to have natural oriented hypergraphic analogs.

\section{Preliminaries and the Matrix-tree Theorem}

\label{MTT}

\subsection{Oriented Hypergraph Basics}

A condensed collection of definitions are provided in this subsection to
improve readability, for a detailed introduction to the definitions the
reader is referred to \cite{OHSachs}. An \emph{oriented hypergraph} is a
quintuple $(V,E,I,\iota ,\sigma )$ where $V$, $E$, and $I$ are disjoint sets
of \emph{vertices}, \emph{edges}, and \emph{incidences}, with \emph{%
incidence function} $\iota :I\rightarrow V\times E$, and \emph{incidence
orientation function} $\sigma :I\rightarrow \{+1,-1\}$. A \emph{bidirected
graph} is an oriented hypergraph with the property that for each $e\in E$, $%
\left\vert \{i\in I\mid (proj_{E}\circ \iota )(i)=e\}\right\vert =2$, and
can be regarded as an orientation of a signed graph (see \cite{SG, OSG})
where the \emph{sign of an edge }$e$ is%
\begin{equation*}
sgn(e)=-\sigma (i)\sigma (j)\text{,}
\end{equation*}%
where $i$ and $j$ are the incidences containing $e$.

A \emph{backstep of }$G$ is an embedding of $\overrightarrow{P}_{1}$ into $G$
that is neither incidence-monic nor vertex-monic; a \emph{loop of }$G$ is an
embedding of $\overrightarrow{P}_{1}$ into $G$ that is incidence-monic but
not vertex-monic; a \emph{directed adjacency of }$G$ is an embedding of $%
\overrightarrow{P}_{1}$ into $G$ that is incidence-monic. A \emph{directed
weak walk of length }$k$\emph{\ in }$G$ is the image of an incidence
preserving embedding of a directed path of length $k$ into $G$.
Conventionally, a backstep is a sequence of the form $(v,i,e,i,v)$, a loop
is a sequence of the form $(v,i,e,j,v)$, and an adjacency is a sequence of
the form $(v,i,e,j,w)$ where $\iota (i)=(v,e)$ and $\iota (j)=(w,e)$. The 
\emph{opposite} embedding is image of the reversal of the initial directed
path, while the non-directed version is the set on the sequence's image.

The \emph{sign of a weak walk} is defined as 
\begin{equation*}
sgn(W)=(-1)^{k}\prod_{h=1}^{2k}\sigma (i_{h})\text{,}
\end{equation*}%
which implies that for a path in $G$ the product of the adjacency signs of
the path is equal to the sign of the path calculated as a weak walk.

\subsection{The Matrix-tree Theorem}

It was shown in \cite{AH1} that the $(v,w)$-entry of the oriented incidence
Laplacian are the negative weak walks of length $1$ from $v$ to $w$ minus
the number of positive weak walks of length $1$ from $v$ to $w$. This was
restated in \cite{OHSachs} as the follows:

\begin{theorem}
\label{WWT}The $(v,w)$-entry of $\mathbf{L}_{G}$ is
 $\dsum\limits_{\omega \in \Omega _{1}}-sgn(\omega (\overrightarrow{P}_{1}))$, 
where $\Omega _{1}$ is the set of all
incidence preserving maps $\omega :\overrightarrow{P}_{1}\rightarrow G$ with 
$\omega (t)=v$ and $\omega (h)=w$.
\end{theorem}

A \emph{contributor of }$\emph{G}$ is an incidence preserving function $%
c:\dcoprod\limits_{v\in V}\overrightarrow{P}_{1}\rightarrow G$ with $%
p(t_{v})=v$ and $\{p(h_{v})\mid v\in V\}=V$. A \emph{strong-contributor of }$%
G$ is an incidence-monic contributor of $G$ -- i.e. the backstep-free
contributors of $G$. Let $\mathfrak{C}(G)$ (resp. $\mathfrak{S}(G)$) denote
the sets of contributors (resp. strong contributors) of $G$. In \cite%
{OHSachs} contributors provided a finest count to determine the
permanent/determinant of Laplacian and adjacency matrices and their
characteristic polynomials of any integral matrix as the incidence matrix of
the associated oriented hypergraph. The values $ec(c)$, $oc(c)$, $pc(c)$,
and $nc(c)$ denote the number of even, odd, positive, and negative circles
in the image of contributor $c$. Additionally, the sets $\mathfrak{C}_{\geq
0}(G)$ (resp. $\mathfrak{C}_{=0}(G)$) denote the set of all contributors
with at least $0$ (resp. exactly $0$) backsteps.

\begin{theorem}[\protect\cite{OHSachs}]
\label{PDLA}Let $G$ be an oriented hypergraph with adjacency matrix $\mathbf{%
A}_{G}$ and Laplacian matrix $\mathbf{L}_{G}$, then

\begin{enumerate}
\item $\mathrm{perm}(\mathbf{L}_{G})=\dsum\limits_{c\in \mathfrak{C}_{\geq
0}(G)}(-1)^{oc(c)+nc(c)}$,

\item $\det (\mathbf{L}_{G})=\dsum\limits_{c\in \mathfrak{C}_{\geq
0}(G)}(-1)^{pc(c)}$,

\item $\mathrm{perm}(\mathbf{A}_{G})=\dsum\limits_{c\in \mathfrak{C}%
_{=0}(G)}(-1)^{nc(c)}$,

\item $\det (\mathbf{A}_{G})=\dsum\limits_{c\in \mathfrak{C}%
_{=0}(G)}(-1)^{ec(c)+nc(c)}$.
\end{enumerate}
\end{theorem}

For a $V\times V$ matrix $\mathbf{M}$, let $U,W\subseteq V$, define $[%
\mathbf{M}]_{(U;W\mathbf{)}}$ be the minor obtained by striking out rows $U$
and columns $W$ from $\mathbf{M}$. Let $\mathfrak{C}(U;W;G)$ be the set of
all sub-contributors of $G$ with $c:\dcoprod\limits_{u\in \overline{U}}%
\overrightarrow{P}_{1}\rightarrow G$ with $p(t_{u})=u$ and $\{p(h_{u})\mid
u\in \overline{U}\}=\overline{W}$. Define $\mathfrak{S}(U;W;G)$ analogously
for strong-contributors. Let the values $en(c)$, $on(c)$, $pn(c)$, and $%
nn(c) $ denote the number of even, odd, positive, and negative,
non-adjacency-trivial components (paths or circles) in the image of $c$.

\begin{theorem}
\label{AMMT}Let $G$ be an oriented hypergraph with adjacency matrix $\mathbf{%
A}_{G}$ and Laplacian matrix $\mathbf{L}_{G}$, then

\begin{enumerate}
\item $\mathrm{perm}([\mathbf{L}_{G}]_{(U;W\mathbf{)}})=\dsum\limits_{c\in 
\mathfrak{C}(U;W;G)}(-1)^{on(c)+nn(c)}$,

\item $\det ([\mathbf{L}_{G}]_{(U;W\mathbf{)}})=\dsum\limits_{c\in \mathfrak{%
C}(U;W;G)}\varepsilon (c)\cdot (-1)^{on(c)+nn(c)}$,

\item $\mathrm{perm}([\mathbf{A}_{G}]_{(U;W\mathbf{)}})=\dsum\limits_{c\in 
\mathfrak{S}(U;W;G)}(-1)^{nn(c)}$,

\item $\det ([\mathbf{A}_{G}]_{(U;W\mathbf{)}})=\dsum\limits_{c\in \mathfrak{%
S}(U;W;G)}\varepsilon (c)\cdot (-1)^{en(c)+nn(c)}$.
\end{enumerate}

Where $\varepsilon (c)$ is the number of inversions in the natural bijection
from $\overline{U}$ to $\overline{W}$.
\end{theorem}

The proof of Theorem \ref{AMMT} is analogous to Theorem 4.1.1 in \cite%
{OHSachs} using the bijective definitions of permanent/determinant.

The value of $\varepsilon (c)$ can be modified to count circle and paths
separately, as the circle components simplify identical to the works in \cite%
{OHSachs}, thus part (2) of Theorem \ref{AMMT} can be restated as follows:

\begin{theorem}
\label{AMMTa}Let $G$ be an oriented hypergraph with adjacency matrix $%
\mathbf{A}_{G}$ and Laplacian matrix $\mathbf{L}_{G}$, then%
\begin{equation*}
\det ([\mathbf{L}_{G}]_{(U;W\mathbf{)}})=\dsum\limits_{c\in \mathfrak{C}%
(U;W;G)}\varepsilon ^{\prime }(c)\cdot (-1)^{pc(c)}\cdot (-1)^{op(c)+np(c)}%
\text{.}
\end{equation*}%
Where $\varepsilon ^{\prime }(c)$ is the number of inversions in paths parts
of the natural bijection from $\overline{U}$ to $\overline{W}$.
\end{theorem}

Comparing, the non-zero elements of $\mathfrak{C}(U;W;G)$ are the
Chaiken-type structures in \cite{Seth1} with multiplicities replaced with
backstep maps.

\section{Contributor Structure of Bidirected Graphs}

\label{CMap}

\subsection{Pre-contributors and Incidence Packing}

An oriented hypergraph in which every edge has exactly $2$ incidences is a \emph{bidirected graph}, 
and can be regarded as orientations of signed graphs (see \cite{MR0267898, SG, OSG}). 

Throughout this section, $G$ is a bidirected graph in which every connected
component contains at least one adjacency, and $\overrightarrow{P}%
_{1}$ is a directed path of length $1$. A \emph{pre-contributor
of }$\emph{G}$ is an incidence preserving function $p:\dcoprod\limits_{v\in
V}\overrightarrow{P}_{1}\rightarrow G$ with $p(t_{v})=v$. That is, the
disjoint union of $\left\vert V\right\vert $ copies of $\overrightarrow{P}%
_{1}$ into $G$ such that every tail-vertex labeled by $v$ is mapped to $v$.

Consider a pre-contributor $p$ with $p(t_{v})\neq p(h_{v})$ for vertex $v\in
V$. \emph{Packing a directed adjacency of a pre-contributor }$p$\emph{\ into
a backstep at vertex }$v$ is a pre-contributor $p_{v}$ such that $p_{v}=p$ for all 
$u\in V\smallsetminus v$, and for vertex $v$ 
\begin{eqnarray*}
p((\overrightarrow{P}_{1})_{v}) &=&(v,i,e,j,w)\text{, }i\neq j\text{,} \\
\text{and }p_{v}((\overrightarrow{P}_{1})_{v}) &=&(v,i,e,i,v)\text{.}
\end{eqnarray*}%
Thus, the head-incidence and head-vertex of adjacency $p((\overrightarrow{P}%
_{1})_{v})$ are identified to the tail-incidence and tail-vertex.

\emph{Unpacking a backstep of a pre-contributor }$p$\emph{\ into an adjacency out of vertex }$%
v $\emph{\ }is a pre-contributor $p^{v}$ is defined analogously but for
vertex $v$, the head-incidence and head-vertex of backstep $p((%
\overrightarrow{P}_{1})_{v})$ are identified to the unique incidence and
vertex that would complete the adjacency in bidirected graph $G$. Note that
this is unique for a bidirected graph since every edge has exactly two
incidences, but this is not the case in if there are edges of size greater
than $2$.

For a bidirected graph $G$ and vertex $v$, let $\mathfrak{P}(G)$ be the set
of all pre-contributors of $G$, $\mathfrak{P}_{v}(G)$ be the set of
pre-contributors with a backstep at $v$, and let $\mathfrak{P}^{v}(G)$ be
the set of pre-contributors with a directed adjacency from $v$.

\begin{lemma}
\label{PUInverses}Packing and unpacking are inverses between $\mathfrak{P}%
_{v}$ and $\mathfrak{P}^{v}$.
\end{lemma}

\begin{proof}
By definition $(p_{v})^{v}=p$ and $(p^{v})_{v}=p$ for appropriate
contributors in $\mathfrak{P}^{v}$ or $\mathfrak{P}_{v}$.
\qed \end{proof}

\begin{lemma}
\label{PCommutes}Packing is commutative.
\end{lemma}

\begin{proof}

Let $p\in \mathfrak{P}^{v}\cap \mathfrak{P}^{w}$, $p_{vw}:=p_{w}\circ p_{v}$%
, and $p_{wv}:=p_{v}\circ p_{w}$. By definition, $p_{vw}=p_{wv}$ for all $(%
\overrightarrow{P}_{1})_{u}$ with $u\in V\smallsetminus \{v,w\}$. For
vertices $v$ and $w$, 
\begin{eqnarray*}
p_{v}((\overrightarrow{P}_{1})_{w}) &=&p((\overrightarrow{P}_{1})_{w})\text{,%
} \\
\text{and }p_{w}((\overrightarrow{P}_{1})_{v}) &=&p((\overrightarrow{P}%
_{1})_{v})\text{.}
\end{eqnarray*}%
Giving, 
\begin{eqnarray*}
p_{vw}((\overrightarrow{P}_{1})_{w}) &=&p_{w}((\overrightarrow{P}%
_{1})_{w})=p_{wv}((\overrightarrow{P}_{1})_{w})\text{,} \\
\text{and }p_{vw}((\overrightarrow{P}_{1})_{v}) &=&p_{v}((\overrightarrow{P}%
_{1})_{v})=p_{wv}((\overrightarrow{P}_{1})_{v})\text{.}
\end{eqnarray*}
\qed \end{proof}

\begin{lemma}
\label{UCommutes}Unpacking is commutative.
\end{lemma}

\begin{proof}

Proof is identical to packing after reversing subscript and superscripts.
\qed \end{proof}

\subsection{Contributors and Activation}

A contributor of $G$ is a pre-contributor where $\{p(h_{v})\mid v\in V\}=V$.

For each $c\in \mathfrak{C}(G)$ let $tc(c)$ be the total number of circles
in $c$; a degenerate $2$-circle (a closed $2$-weak-walk) is considered a
circle, while a degenerate $1$-circle (a backstep) is not. \emph{Activating
a circle of contributor }$c$ is a minimal sequence of unpackings that
results in a new contributor $c^{\prime }$ such that $tc(c)=tc(c^{\prime })-1$. 
\emph{Deactivating a circle of contributor }$c$ is a minimal sequence of
packings that results in a new contributor $c^{\prime \prime }$ such that $%
tc(c)=tc(c^{\prime \prime })+1$. Immediately from the definition we have:

\begin{lemma}
Let $c,d\in \mathfrak{C}(G)$. Contributor $d$ can be obtained by activating
circles of $c$ if, and only if, $c$ can be obtained by deactivating a
circles of $d$. Moreover, the activation/deactivation sets are equal.
\end{lemma}

Define the \emph{activation partial order} $\leq _{a}$ where $c\leq _{a}d$ if 
$d$ is obtained by a sequence of activations starting with $c$, and the 
\emph{activation equivalence relation} $\sim _{a}$where $c\sim _{a}d$ if $%
c\leq _{a}d$ or $d\leq _{a}c$. The elements of $\mathfrak{C}(G)/\sim _{a}$
are called the \emph{activation classes of }$G$.

\begin{example}
\label{Ex1}Figure \ref{ContAre} is a bidirected graph (with incidences omitted), and some contributors are depicted, sorted by their associated permutation.
\begin{figure}[H]
\centering\includegraphics[scale=1]{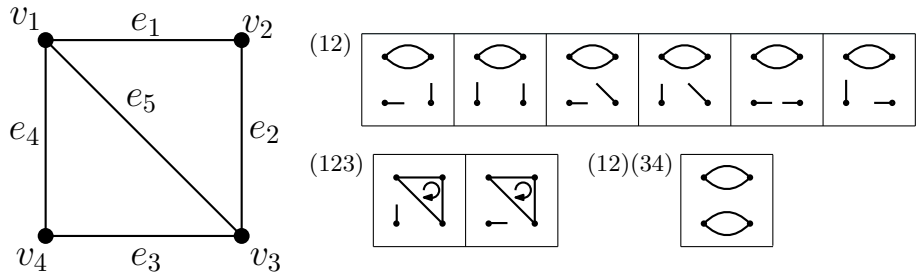}
\caption{Understanding contributors.}
\label{ContAre}
\end{figure}
Observe that the fifth contributor in $(12)$ is below the $(12)(34)$ contributor in the activation partial order.
\end{example}

\begin{lemma}
The minimal elements of activation classes are incomparable, consist of only
backsteps, and correspond to the identity permutation.
\end{lemma}

\begin{lemma}
\label{IsBoolean}All activation classes of $G$ are boolean lattices.
\end{lemma}

\begin{proof}

For a given activation class consider the set of possible active circles.
The elements of each activation class is ordered by the subsets of active
circles, with unique maximal element having all circles active, and unique
minimal element having all circle inactive.
\qed \end{proof}

We have the following lemma using the facts that (1) every connected
component of $G$ is assumed to have an adjacency, and (2) every contributor
corresponds to a permutation on the vertices, there is at least one circle
that can be activated for each minimal element in each activation class.

\begin{lemma}
Each maximal contributor in a activation class contains at least $1$ circle.
\end{lemma}

\begin{corollary}
\label{NonTrivialCC}Each activation class has at least $2$ members.
\end{corollary}

\begin{example}
\label{Ex2}Figure \ref{Cclasses} shows $3$ activation classes of the graph from Example \ref{Ex1}. 
\begin{figure}[H]
\centering\includegraphics[scale=1]{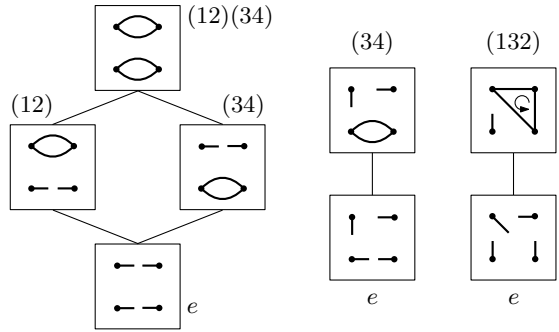}
\caption{Activation classes are boolean.}
\label{Cclasses}
\end{figure}
The activation classes are ranked by the number of circles, and the minimal element corresponds to the identity permutation.
\end{example}

\subsection{Partitioning Activation Classes}

\label{Cuts}

For $u,w\in V$ two contributors $c$ and $d$ are $uw$\emph{-equivalent},
denoted $c\sim _{uw}d$, if $c(h_{u})=d(h_{u})=w$. Since $\sim _{uw}$ only
collects contributors in which the image of $(\overrightarrow{P}_{1})_{u}$
has head-vertex mapped to $w$ we have:

\begin{lemma}
$\mathfrak{C}(G)/(\sim _{u_{1}w_{1}}\circ \sim _{u_{2}w_{2}})=\mathfrak{C}%
(G)/(\sim _{u_{2}w_{2}}\circ \sim _{u_{1}w_{1}})$ for $%
u_{1},w_{1},u_{2},w_{2}\in V$.
\end{lemma}

As $w$ varies, the composition $\sim _{u\bullet }:=\mathop\bigcirc
\limits_{w\in V}\sim _{uw}$ is well defined without the need for a total
ordering on $V$. Moreover,

\begin{lemma}
$\mathfrak{C}(G)/(\sim _{uw}\circ \sim _{a})=\mathfrak{C}(G)/(\sim _{a}\circ
\sim _{uw})$ or any $u,w\in V$.
\end{lemma}

By construction, the relation $\sim _{uw}\circ \sim _{a}$ (subsequently, $%
\sim _{u\bullet }\circ \sim _{a}$) refines each $\mathcal{A}\in \mathfrak{C}%
(G)/\sim _{a}$. Let $\mathcal{A}/\sim _{uw}$\ and $\mathcal{A}/\sim
_{u\bullet }$ denote the refinement of $\mathcal{A}$ by $\sim _{uw}$ or $%
\sim _{u\bullet }$ respectively.

\begin{theorem}
$\mathfrak{C}(G)/(\sim _{u\bullet }\circ \sim _{a})$ is a refinement of each 
activation class in $\mathfrak{C}(G)/\sim _{a}$ into two principal order
ideals (one upper and one lower, with the upper order ideal possibly empty)
that are boolean complements. Moreover, the upper order ideal of activation
class $\mathcal{A}$ is empty if, and only if, $c(h_{u})=u$ for all $c\in 
\mathcal{A}$.
\end{theorem}

\begin{proof}

Let $\mathcal{A}\in \mathfrak{C}(G)/\sim _{a}$ and observe that every least
element of $\mathcal{A}$ is an adjacency free contributor, so the set of
possible maximal elements such that $c(h_{u})=u$ is non-empty. Using the
definition of activation, the facts that $\mathcal{A}$ is boolean, and that
there is at least one element (the $\mathbf{0}$-element) in each activation
class with $h_{u}\rightarrow u$, there is exactly one maximal element with $%
h_{u}\rightarrow u$, let $M(u;u;\mathcal{A})$ be this maximal element. Thus,
the principal ideal $\downarrow M(u;u;\mathcal{A})$ exists and is
necessarily boolean. Moreover, $\downarrow M(u;u;\mathcal{A})=\mathcal{A}$
if, and only if, $c(h_{u})=u$ for all contributors $c\in \mathcal{A}$, thus $%
\mathcal{A}/\sim _{uw}$ is empty for all $w\neq u$.

Since $\mathcal{A}$ is boolean, if there is a contributor $d$ such that $%
d(h_{u})=w\neq u$, then all contributors of $\mathcal{A}$ with $%
h_{u}\nrightarrow u$ must have $h_{u}\rightarrow w$, since every edge is a $%
2 $-edge. Moreover, if there is a contributor of $\mathcal{A}$ with $%
h_{u}\rightarrow w$, then there is a unique minimal element with $%
m(h_{u})=w\neq u$, let $m(u;w;\mathcal{A})$ be this minimal element (if it
exists). By construction, $m(u;w;\mathcal{A})$ is the contributor of $%
\mathcal{A}$ with only the circle containing the $uw$-adjacency active, is a
rank $1$ element in $\mathcal{A}$, and is the boolean complement of $M(u;u;%
\mathcal{A})$. Thus, $\mathcal{A}=\downarrow M(u;u;\mathcal{A})\cup \uparrow
m(u;w;\mathcal{A})$.
\qed \end{proof}

The $(u;w)$\emph{-cut of activation class} $\mathcal{A}$ is the subclass of $%
\mathcal{A}/\sim _{uw}$ where each element has $c(h_{u})=w$ --- that is, $%
\downarrow M(u;u;\mathcal{A})$ if $u=w$, or $\uparrow m(u;w;\mathcal{A})$ if 
$u\neq w$ and $m(u;w;\mathcal{A})$ exists. Let $U,W\subseteq V$ with $%
\left\vert U\right\vert =\left\vert W\right\vert =k$, and $\mathbf{u}%
=(u_{1},u_{2},\ldots ,u_{k})$, $\mathbf{w}=(w_{1},w_{2},\ldots ,w_{k})$ be
linear orderings of $U$ and $W$ according to their placement in the implied
linear ordering of $V$ in the underlying incidence matrix. The $(\mathbf{u};%
\mathbf{w})$\emph{-cut of the activation class }$\mathcal{A}$ is the
corresponding subclass in $\mathcal{A}/\mathop\bigcirc \limits_{i\in \lbrack
k]}\sim _{u_{i}w_{i}}$. Let $\mathcal{A}(\mathbf{u};\mathbf{w};G)$ denote
the $(\mathbf{u};\mathbf{w})$-cut of activation class $\mathcal{A}$, and $%
\widehat{\mathcal{A}}(\mathbf{u};\mathbf{w};G)$ be the elements of $\mathcal{%
A}(\mathbf{u};\mathbf{w};G)$ with the adjacency or backstep from $u_{i}$ to $%
w_{i}$ is removed for each $i$. Let $\mathfrak{C}(\mathbf{u};\mathbf{w};G)$
be the set of all elements in all $\mathcal{A}(\mathbf{u};\mathbf{w};G)$,
and $\widehat{\mathfrak{C}}(\mathbf{u};\mathbf{w};G)$ be the elements of $%
\mathfrak{C}(\mathbf{u};\mathbf{w};G)$ with the adjacency or backstep from $%
u_{i}$ to $w_{i}$ is removed for each $i$.

\begin{example}
\label{Ex3}Figure \ref{C classes circled} shows $(v_{1},v_{1})$-cuts of the
contribution classes from Figure \ref{Cclasses}. 
\begin{figure}[H]
\centering
\includegraphics[scale=1]{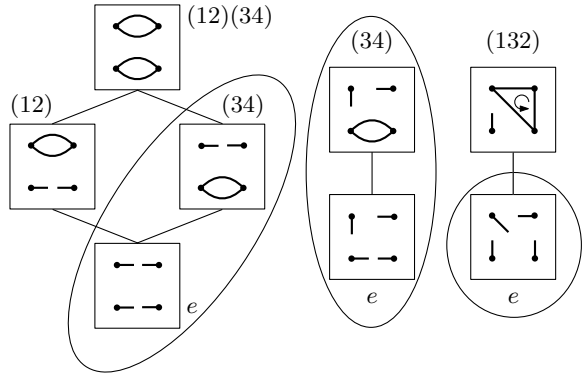}
\caption{$(v_{1},v_{1})$-cuts of contribution classes.}
\label{C classes circled}
\end{figure}

Observe that the first two sub-classes are non-trivial boolean lattices, the
final sub-class is a trivial boolean lattices, and the second sub-class
has an empty upper order ideal.
\end{example}

\section{Universal Contributors and The Matrix-tree Theorem}

\label{MTT2}

\subsection{Adjacency Completion}

Given an oriented hypergraph $G=(V,E,I,\iota ,\sigma )$ let $G^{\prime
}=(V,E\cup E_{0},I\cup I_{0},\iota ,\sigma ^{\prime })$ be the oriented
hypergraph obtained by adding a bidirected edge to $G$ for every
non-adjacent pair of vertices, where $\sigma ^{\prime }=\sigma $ for all $%
i\in I$, and $\sigma ^{\prime }=0$ for all $i\in I_{0}$ (see \cite{Grill1}
for relationship to the injective envelope). The \emph{sign of a
(sub-)contributor} is the product of the weak walks. The inclusion of $0$%
-signed-incidences in $G^{\prime }$ implies that an element of $\widehat{%
\mathfrak{C}}(\mathbf{u};\mathbf{w};G)$ has non-zero sign if, and only if,
it exists in $G$. Let $\widehat{\mathfrak{C}}_{\neq 0}(\mathbf{u};\mathbf{w}%
;G^{\prime })$ be the set of non-zero elements of $\widehat{\mathfrak{C}}(%
\mathbf{u};\mathbf{w};G^{\prime })$. This fact gives the following simple
Lemma that relates the global contributors of $G^{\prime }$ to the
Chaiken-type forests of \cite{Seth1} separated by multiplicity:

\begin{lemma}
\label{ChaikenFigs}
If $U,W\subseteq V$ with with linear orderings $\mathbf{u}$ and $\mathbf{w}$%
, then 
\begin{eqnarray*}
\widehat{\mathfrak{C}}_{\neq 0}(\mathbf{u};\mathbf{w};G^{\prime })=%
\mathfrak{C}(U;W;G).
\end{eqnarray*}
\end{lemma}

\begin{example}
\label{Ex4}Figure \ref{AdjComplCont} shows an additional contribution class 
from Figure \ref{Cclasses} that exists in the adjacency completion.
\begin{figure}[H]
\centering
\includegraphics[scale=1]{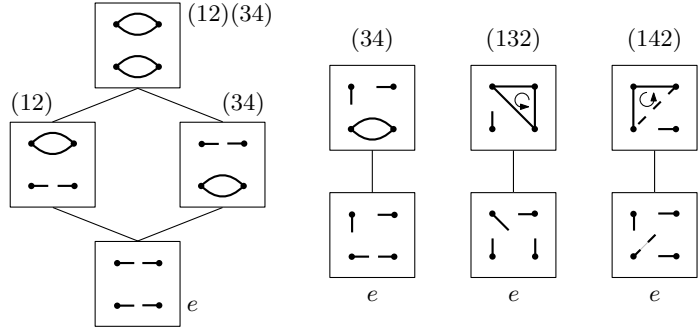}
\caption{Contributors in $G^{\prime}$.}
\label{AdjComplCont}
\end{figure}
Observe that the $(142)$ contributor with the $v_4v_2$-adjacency removed  is a member 
of $\widehat{\mathfrak{C}}_{\neq 0}(v_4;v_2;G^{\prime })$, exists in $G$, and counts in 
the $v_4v_2$-minor calculation.
\end{example}

Lemma \ref{ChaikenFigs} provides the following restatement of Theorem \ref{AMMT}:

\begin{theorem}
Let $G$ a bidirected graph with Laplacian matrix $\mathbf{L}_{G}$. Given $%
U,W\subseteq V$ with $\left\vert U\right\vert =\left\vert W\right\vert $ and
linear orderings $\mathbf{u}$ and $\mathbf{w}$, let $[\mathbf{L}_{G}]_{(%
\mathbf{u};\mathbf{w)}}$ be the minor of $\mathbf{L}_{G}$ formed by the
ordered deletion of the rows corresponding to the vertices in $U$ and the
columns corresponding to the vertices in $W$. Then we have,

\begin{enumerate}
\item $perm([\mathbf{L}_{G}]_{(\mathbf{u};\mathbf{w)}})=\dsum\limits_{c\in 
\widehat{\mathfrak{C}}_{\neq 0}(\mathbf{u};\mathbf{w};G^{\prime
})}(-1)^{on(c)+nn(c)}$,

\item $\det ([\mathbf{L}_{G}]_{(\mathbf{u};\mathbf{w)}})=\dsum\limits_{c\in 
\widehat{\mathfrak{C}}_{\neq 0}(\mathbf{u};\mathbf{w};G^{\prime
})}\varepsilon (c)\cdot (-1)^{on(c)+nn(c)}$.
\end{enumerate}

Where $\varepsilon (c)$ is the number of inversions in the natural bijection
from $\overline{U}$ to $\overline{W}$.
\end{theorem}

\subsection{Applications}

We now examine alternate proofs of established results using the boolean
order of contributor maps. Besides providing more insight into contributors,
the hope is these techniques can be generalized to a complete theory for
oriented hypergraphs --- as evidenced by Theorem \ref{AMMT}.

Let $\mathfrak{M}^{-}$ be the set of maximal elements from the positive-circle-free 
activation classes. 

\begin{lemma}
If $G$ is a signed graph, then $\det (\mathbf{L}_{G})=\dsum\limits_{c\in 
\mathfrak{M}^{-}}2^{nc(c)}$.
\end{lemma}

\begin{proof}

From Theorem \ref{PDLA} and Lemma \ref{IsBoolean}%
\begin{eqnarray*}
\det (\mathbf{L}_{G}) &=&\dsum\limits_{c\in \mathfrak{C}_{\geq
0}(G)}(-1)^{pc(c)} \\
&=&\dsum\limits_{\mathcal{A}\in \mathfrak{C}(G)/\sim _{a}}\dsum\limits_{c\in 
\mathcal{A}}(-1)^{pc(c)}\text{.}
\end{eqnarray*}%
Let $M_{\mathcal{A}}^{-}$ be the minimal element of activation class $%
\mathcal{A}$ that has the maximal number of negative circles (if it exists). 

\textit{Case 1 ($M_{\mathcal{A}}^{-}$ does not exist):} If $M_{\mathcal{A}}^{-}$ does 
not exist, then $\mathcal{A}$ is a boolean lattice with every circle positive. Since each 
contributor is signed $(-1)^{pc(c)}$, the signs of the contributors of $\mathcal{A}$ alternate with rank. Thus, 
the sum of elements is the alternating sum of binomial coefficients which equals $0$. 

\textit{Case 2 ($M_{\mathcal{A}}^{-}$ exists):} If $M_{\mathcal{A}}^{-}$ exists, it is 
necessarily unique, $\downarrow M_{\mathcal{A}}^{-}$ is boolean, and every element 
of $\downarrow M_{\mathcal{A}}^{-}$ is signed $+1$. 

\textit{Case 2a (no positive circles):} 
If there are no positive circles in $\mathcal{A}$, then $\downarrow M_{\mathcal{A}}^{-} = \mathcal{A}$ and 
the sum of the elements is $2^{nc(M_{\mathcal{A}})}$.

\textit{Case 2b (at least one positive circle):} Suppose $1_\mathcal{A}$ has exactly $n$ positive circles, 
let $p_1, p_2, \ldots, p_n$ be contributors of $\mathcal{A}$ with exactly one positive circle, 
and let $P_i=[0_{\mathcal{A}},p_i]$. For $k \in [0,n]$, consider the collection 
of boolean lattices $$B_k = (\downarrow M_{\mathcal{A}}) \vee \bigvee\limits_{i=1}^{k}{P_i},$$ 
where $B_0 = \downarrow M_{\mathcal{A}}$. By construction, $B_i \cong B_i \vee p_{i+1}$, 
$B_{i+1} = B_{i} \vee P_{i+1}$, $B_{i+1}$ is a boolean lattice 1 rank larger than $B_{i}$, and 
$B_n = \mathcal{A}$. Additionally, for each $b \in B_{i}$ and $b \vee p_{i+1} \in B_i \vee p_{i+1}$, 
$$(-1)^{pc(b \vee p_{i+1})} = (-1)^{pc(b)+1}.$$ Thus, the sum of all contributors is necessarily $0$ as long 
as there exists a positive circle.

The only non-cancellative activation classes are those that are positive-circle-free, and the result follows 
with the observation that an empty sum is $0$. \qed \end{proof}

\begin{corollary}
If $G$ is a balanced signed graph, then $\det (\mathbf{L}_{G})=0$.
\end{corollary}

Let $\widehat{\mathcal{A}}_{\neq 0}(u;w;G^{\prime })$ be the set of non-zero
elements of $\widehat{\mathcal{A}}(u;w;G^{\prime })$.

\begin{lemma}
If $G$ is a bidirected graph, then the set of elements in all single-element $\widehat{%
\mathcal{A}}_{\neq 0}(u;w;G^{\prime })$ is activation equivalent to the set
of spanning trees of $G$.
\end{lemma}

\begin{proof}

\textit{Case 1a (}$u=w$\textit{):} Suppose $u=w$, and let $\widehat{\mathcal{%
A}}_{\neq 0}(u;u;G^{\prime })$ be a single-element activation classes of $%
G^{\prime }$. The element of $\widehat{\mathcal{A}}_{\neq 0}(u;u;G^{\prime
}) $ consists of exactly $\left\vert V\right\vert -1$ backsteps, all of
which exist in $G$, but none of which contain $u$. Unpacking all backsteps
results in a circle-free subgraph on $\left\vert V\right\vert $ vertices
with $\left\vert V\right\vert -1$ edges, i.e. a spanning tree --- if it was
not circle-free then $\widehat{\mathcal{A}}_{\neq 0}(u;u;G^{\prime })$ would
have more than a single element. Thus, the cardinality of the set of
single-element $\widehat{\mathcal{A}}_{\neq 0}(u;u;G^{\prime })$ is less
than, or equal to, the number spanning trees of $G$.

\textit{Case 1b (}$u=w$\textit{):} Now consider all spanning outward
arborescences of $G$ rooted at $u$. For each arborescence, pack all
adjacencies along the opposite orientation of the arborescence to produce a
unique, non-zero, element of an $\widehat{\mathcal{A}}(u;u;G^{\prime })$.
Thus, the cardinality of the set of single-element $\widehat{\mathcal{A}}%
_{\neq 0}(u;u;G^{\prime })$ is greater than, or equal to, the number
spanning trees of $G$.

\textit{Case 2a (}$u\neq w$)\textit{:} Suppose $u\neq w$, and let $\widehat{%
\mathcal{A}}_{\neq 0}(u;w;G^{\prime })$ be a single-element activation class
of $G^{\prime }$. Since $\widehat{\mathcal{A}}_{\neq 0}(u;w;G^{\prime })$ is
obtained by the upper order ideal of $\mathcal{A}(u;w;G^{\prime })$
generated by the maximal contributor $M(u;w;\mathcal{A})$, and $\widehat{%
\mathcal{A}}_{\neq 0}(u;w;G^{\prime })$ consists of a single element, $M(u;w;%
\mathcal{A})$ must be unicyclic. Thus, the corresponding contributor in $%
\widehat{\mathcal{A}}_{\neq 0}(u;w;G^{\prime })$ must contain a $wu$-path on 
$k$ vertices, and backsteps at the $\left\vert V\right\vert -k$ vertices
outside the path. Unpacking all backsteps is a circle-free subgraph on $%
\left\vert V\right\vert $ vertices and $\left\vert V\right\vert -1$ edges
--- if not then $\widehat{\mathcal{A}}_{\neq 0}(u;w;G^{\prime })$ would have
more than a single element. Thus, the cardinality of the set of
single-element $\widehat{\mathcal{A}}_{\neq 0}(u;w;G^{\prime })$ is less
than, or equal to, the number spanning trees of $G$.

\textit{Case 2b (}$u\neq w$)\textit{: }Now consider the set of spanning
trees of $G$. To see that the cardinality of the set of single-element $%
\widehat{\mathcal{A}}_{\neq 0}(u;w;G^{\prime })$ is greater than, or equal
to, the number spanning trees of $G$, examine the following sub-cases:

\textit{Case 2b, part 1 (}$u\neq w$)\textit{: }If a spanning tree $T$
contains an adjacency between $u$ and $w$ construct the
unicyclic-contributor of $G$ as follows: (1) add another parallel adjacency
between $u$ and $w$, (2) orient the parallel edges to form a degenerate $2$%
-circle, and (3) pack all remaining adjacencies away from $u$ and $w$. The
member of $\widehat{\mathcal{A}}_{\neq 0}(u;w;G^{\prime })$ is obtained by
deleting the $uw$-directed-adjacency.

\textit{Case 2b, part 2 (}$u\neq w$)\textit{:} If a spanning tree $T$ does
not contain an adjacency between $u$ and $w$ construct the
unicyclic-contributor of $G$ as follows: (1) add a $uw$-directed-adjacency,
(2) oriented the resulting unique fundamental circle coherently with the $uw$%
-directed-adjacency, and (3) pack all remaining adjacencies away from the
fundamental circle. The member of $\widehat{\mathcal{A}}_{\neq
0}(u;w;G^{\prime })$ is obtained by deleting the $uw$-directed-adjacency.
\qed \end{proof}

Using the techniques of the previous two Lemmas, adjusting for the cofactor,
and using the fact that every adjacency (hence, circle) is positive in a graph, it is easy to reclaim:

\begin{corollary}
If $G$ is a graph, then the $uw$-cofactor of $\mathbf{L}_{G}$ is $T(G)$, the
number of spanning trees of $G$.
\end{corollary}

\section*{Acknowledgements}

This research is partially supported by Texas State University Mathworks.
The authors sincerely thank the referee for careful reading the manuscript
and for their valuable feedback.



\bibliographystyle{amsplain2}
\bibliography{mybib}

\end{document}